\documentclass{amsart}

\usepackage{amsmath}
\usepackage{amssymb}
\usepackage{bibentry}
\usepackage{enumitem}
\usepackage{graphicx}
\usepackage{mathrsfs}
\usepackage{tabu}
\usepackage{tikz}
\usepackage[comma,numbers,sort,square]{natbib}
\usepackage[colorlinks=true,
            linktocpage=true,
            linkcolor=magenta,
            citecolor=magenta,
            urlcolor=magenta]{hyperref}
\PassOptionsToPackage{hyphens}{url}
\usepackage{cleveref}

\usetikzlibrary{arrows}
\usetikzlibrary{shapes}
\usetikzlibrary{snakes}

\newtheorem{theorem}{Theorem}[section]
\newtheorem{proposition}[theorem]{Proposition}

\newtheorem{corollary}[theorem]{Corollary}
\theoremstyle{definition}
\newtheorem{definition}[theorem]{Definition}

\newtheorem{question}[theorem]{Question}

\DeclareMathOperator{\res}{\upharpoonright}
\newcommand{\ran}{\operatorname{ran}}
\newcommand{\dom}{\operatorname{dom}}

\newcommand{\seq}[1]{\langle #1 \rangle}

\newcommand{\set}[1]{\{ #1 \}}

\newcommand{\PA}{\mathsf{PA}}
\newcommand{\RCA}{\mathsf{RCA}}
\newcommand{\ACA}{\mathsf{ACA}}
\newcommand{\WKL}{\mathsf{WKL}}

\newcommand{\RT}{\mathsf{RT}}
\newcommand{\COH}{\mathsf{COH}}

\newcommand{\SRT}{\mathsf{SRT}}

\newcommand{\WWKL}{\mathsf{WWKL}}
\newcommand{\DNR}{\mathsf{DNR}}

\newcommand{\TS}{\mathsf{TS}}

\newcommand{\cred}{\leq_{\text{\upshape c}}}

\newcommand{\uequiv}{\equiv_{\text{\upshape W}}}
\newcommand{\ured}{\leq_{\text{\upshape W}}}
\newcommand{\nured}{\nleq_{\text{\upshape W}}}
\newcommand{\suequiv}{\equiv_{\text{\upshape sW}}}
\newcommand{\sured}{\leq_{\text{\upshape sW}}}
\newcommand{\nsured}{\nleq_{\text{\upshape sW}}}
\newcommand{\Tred}{\leq_{\text{\upshape T}}}

\newcommand{\fopart}{{}^1}
\newcommand{\PROB}{\mathcal{P}}
\newcommand{\FPROB}{\mathcal{F}}

\newcommand{\C}{\mathsf{C}}

\newcommand{\Id}{\mathsf{Id}}

\begin{document}

\title{On the first-order parts of problems in the Weihrauch degrees}

\author{Damir D. Dzhafarov}
\address{Department of Mathematics\\
University of Connecticut\\
Storrs, Connecticut U.S.A.}
\email{damir@math.uconn.edu}

\author{Reed Solomon}
\address{Department of Mathematics\\
University of Connecticut\\
Storrs, Connecticut U.S.A.}
\email{solomon@math.uconn.edu}

\author{Keita Yokoyama}
\address{Mathematical Institute\\
Tohoku University\\
Sendai, Japan.}
\email{keita.yokoyama.c2@tohoku.ac.jp}

\thanks{Dzhafarov and Solomon were partially supported by a Focused Research Group grant from the National Science Foundation of the United States, DMS-1854355. Yokoyama was partially supported by JSPS KAKENHI grants, numbers JP19K03601 and JP21KK0045. The authors thank Alberto Marcone, Ludovic Patey, and Manlio Valenti for helpful comments during the preparation of this paper.}

\begin{abstract}
	We introduce the notion of the \emph{first-order part} of a problem in the Weihrauch degrees. Informally, the first-order part of a problem $\mathsf{P}$ is the strongest problem with codomaixn $\omega$ that is Weihrauch reducible to $\mathsf{P}$. We show that the first-order part is always well-defined, examine some of the basic properties of this notion, and characterize the first-order parts of several well-known problems from the literature.
\end{abstract}

\maketitle

\section{Introduction}\label{sec:intro}

This paper continues the investigation of the interface between reverse mathematics and computable analysis. The connection was first suggested by Gherardi and Marcone \cite{GM-2008}, and later independently by Dorais, Dzhafarov, Hirst, Mileti, and Shafer \cite{DDHMS-2016}. In the past few years, this area has blossomed into a rich and productive area of research, with by now many papers dedicated to it.

Reverse mathematics uses a combined computability-theoretic and proof-theoretic point of view to analyze the logical strength of theorems that can be formalized in second-order arithmetic. Much of this focuses on \emph{second-order} strength, meaning more specifically, on which set-existence axioms are necessary and sufficient to prove a given such theorem. But there has also been a great deal of work on \emph{first-order} strength, meaning on which number-theoretic results are derivable from a particular theorem. The strongest such result is commonly referred to as the \emph{first-order part} of the theorem. It is an impressive fact that theorems entirely about \emph{sets} of natural numbers (and by extension, mathematical objects that can be represented by sets of numbers) can have nontrivial and often surprising first-order parts.

Computably analysis, on the other hand, is concerned with computational problems rather than theorems, and on how these relate to one another under various reducibilities, the most important of which is \emph{Weihrauch reducibility}. As we review in detail below, there is a well-known correspondence between such problems and a common type of theorem analyzed in reverse mathematics, and this connection has been strengthened by some of the recent investigations mentioned above.

We refer the reader to Simpson \cite{Simpson-2009}, Hirschfeldt \cite{Hirschfeldt-2014}, and Dzhafarov and Mummert \cite{DM-2022} for background in reverse mathematics. We include some basic definitions from computable analysis below, and refer the reader to Brattka and Pauly \cite{BP-TA} and Brattka, Gherardi, and Pauly \cite{BGP-TA} for more complete introductions.

In this paper, we introduce the notion of the \emph{first-order part of a problem}. By analogy with the reverse mathematics setting, this is informally meant to capture the strongest ``number-theoretic'' problem that is Weihrauch reducible to that problem. We will explore the connections between our definition and the first-order parts of theorems in reverse mathematics. We will then classify, as case studies, the first-order parts of several major problems from the literature. We note that, since the circulation of a preprint of this paper, there have been several papers continuing the investigation of the first-order parts of problems, including by Sold\`{a} and Valenti \cite{SV-TA}, Goh, Pauly, and Valenti \cite{GPV-2021}, and Cipriani, Marcone, and Valenti \cite{CMV-TA}.

We begin, in this section, by laying down some background for the rest of the paper. Our computability-theoretic notation and terminology is largely standard, following, e.g., Soare \cite{Soare-2016} or Downey and Hirschfeldt \cite{DH-2010}. We highlight some small departures from this below.

We embed $\omega$ in $\omega^{\omega}$ by identifying $x \in \omega$ with $f \in \omega^\omega$ such that $f(0) = x$ and $f(y) = 0$ for all $y > x$. Throughout, \emph{set} will mean subset of $\omega$ unless otherwise noted. We identify sets with their characteristic functions, so that all sets are elements of $2^\omega \subseteq \omega^\omega$. For $f,g \in \omega^\omega$, we write $\seq{f,g}$ for the (Turing) join of $f$ and $g$. If $f \in \omega^\omega$ and $\sigma \in \omega^{<\omega}$, then $\seq{f,\sigma}$ refers to the string $\tau$ of length $2|\sigma|$ such that $\tau(2x) = f(x)$ and $\tau(2x+1) = \sigma(x)$ for all $x < |\sigma|$. We define $\seq{\sigma,f}$ analogously. Joins of more than two objects are handled as usual, with the understanding that the number of objects involved is uniformly computable from the join itself.

We also write $\seq{y_0,\ldots,y_{k-1}}$ for the string $\sigma \in \omega^k$ with $\sigma(x) = y_x$ for all $x < k$. For $b \in \{0,1\}$, we write $b^k$ for the string $\sigma \in \omega^k$ with $\sigma(x) = b$ for all $x < k$, and $b^\omega$ for the sequence $f \in \omega^\omega$ with $f(x) = b$ for all $x$. For $\sigma,\tau \in \omega^{<\omega}$, we write $\sigma \tau$ or $\sigma^\frown{}\tau$ for the concatenation of $\sigma$ by $\tau$, and similarly for the concatenation of finite strings by infinite sequences.

We regard all Turing functionals $\Phi$ as partial continuous maps $\omega^\omega \to \omega^\omega$, and write $\Phi(f) = g$ if $\Phi^f(x) \downarrow = g(x)$ for all $x \in \omega$. In this paper, we also write $\Phi(f)(x)$ in place of $\Phi^f(x)$ and $\Phi(f,g)$ in place of $\Phi(\seq{f,g})$, etc. We follow the convention that if $\Phi(f)(x) \downarrow$ for some $f \in \omega^\omega$ and some $x \in \omega$ then $\Phi(f)(y) \downarrow$ for all $y \leq x$. For $\sigma \in \omega^{<\omega}$, we write $\Phi(\sigma)(x) \downarrow = y$ if $\Phi(f)(x) \downarrow = y$ for all $f \in \omega^\omega$ extending $\sigma$ in at most $|\sigma|$ many steps. In particular, if $\Phi(f)(x) \downarrow$ then $\Phi(f \res k)(x) \downarrow$ for some $k \in \omega$.

We use uppercase Greek letters for arbitrary Turing functionals, $\Phi,\Delta,\Gamma$, etc., and let $\Phi_0,\Phi_1,\ldots$ denote a fixed computable listing of all Turing functionals.

\begin{definition}
\
	\begin{enumerate}
		\item A \emph{problem} is a partial multifunction $\mathsf{P} : \subseteq \omega^{\omega} \rightrightarrows \omega^\omega$. Each $f \in \dom(\mathsf{P})$ is an \emph{instance of $\mathsf{P}$}, or simply a \emph{$\mathsf{P}$-instance}, and each $g \in \mathsf{P}(f)$ is a \emph{solution to $f$ in $\mathsf{P}$}, or simply a \emph{$\mathsf{P}$-solution to $f$}. We let $\PROB$ denote the class of all problems.
		\item A problem $\mathsf{Q}$ is \emph{Weihrauch reducible} to a problem $\mathsf{P}$, written $\mathsf{Q} \ured \mathsf{P}$, if there exist Turing functionals $\Phi$ and $\Psi$ such that for all $f \in \dom(\mathsf{Q})$ we have $\Phi(f) \in \dom(\mathsf{P})$, and for every $\widehat{g} \in \mathsf{P}(\Phi(f))$ we have $\Psi(f,\widehat{g}) \in \mathsf{Q}(f)$. In this case, we also say $\mathsf{Q}$ is Weihrauch reducible to $\mathsf{P}$ \emph{via} $\Phi$ and $\Psi$.
		\item A problem $\mathsf{Q}$ is \emph{strongly Weihrauch reducible} to a problem $\mathsf{P}$, written $\mathsf{Q} \sured \mathsf{P}$, if there exist Turing functionals $\Phi$ and $\Psi$ such that for all $f \in \dom(\mathsf{Q})$ we have $\Phi(f) \in \dom(\mathsf{P})$, and for every $\widehat{g} \in \mathsf{P}(\Phi(f))$ we have $\Psi(\widehat{g}) \in \mathsf{Q}(f)$. In this case, we say $\mathsf{Q}$ is strongly Weihrauch reducible to $\mathsf{P}$ \emph{via} $\Phi$ and $\Psi$.
	\end{enumerate}
\end{definition}

\noindent The equivalence classes under $\ured$ form the \emph{Weihrauch degrees}.

There is a natural correspondence between problems and theorems of the form $(\forall X) [\varphi(X) \to (\exists Y) \psi(X,Y)]$ for $X,Y \in \omega^\omega$. Any problem gives rise to a theorem of this form by letting $\varphi$ define the instances, and $\psi$ the solutions. Conversely, any theorem of this form gives rise to the problem whose instances are the $X$ satisfying $\varphi(X)$, and the solutions to a given $X$ are the $Y$ satisfying $\psi(X,Y)$. (See \cite[Chapter 4]{DM-2022} for a more detailed discussion.) When $\varphi$ and $\psi$ are arithmetical properties, the theorems of this form are $\Pi^1_2$ statements of second-order arithmetic, which constitute the bulk of theorems studied in reverse mathematics. The specific problems we consider in examples in the sequel will all correspond to $\Pi^1_2$ theorems in this fashion, and we will move back and forth between the problem and theorem perspectives.

For completeness, we include definitions of some of the principal operations on the Weihrauch degrees.

\begin{definition}
	Fix $\mathsf{P},\mathsf{Q} \in \PROB$.
	\begin{enumerate}
		\item $\mathsf{P} \sqcup \mathsf{Q}$ (the \emph{join} of $\mathsf{P}$ and $\mathsf{Q}$) is the problem with domain $\dom(\mathsf{P}) \sqcup \dom(\mathsf{Q})$ and $(\mathsf{P} \sqcup \mathsf{Q})(\seq{0,f}) = \set{0} \times \mathsf{P}(f)$ and $(\mathsf{P} \sqcup \mathsf{Q})(\seq{1,g}) = \set{1} \times \mathsf{Q}(g)$ for all $f \in \dom(\mathsf{P})$ and $g \in \dom(\mathsf{Q})$.
		\item $\mathsf{P} \sqcap \mathsf{Q}$ (the \emph{meet} of $\mathsf{P}$ and $\mathsf{Q}$) is the problem with domain $\dom(\mathsf{P}) \times \dom(\mathsf{Q})$ and $(\mathsf{P} \sqcap \mathsf{Q})(\seq{f,g}) = (\set{0} \times \mathsf{P}(f)) \cup (\set{1} \times \mathsf{Q}(g))$ for all $f \in \dom(\mathsf{P})$ and $g \in \dom(\mathsf{Q})$.
		\item $\mathsf{P} \times \mathsf{Q}$ (the \emph{parallel product} of $\mathsf{P}$ and $\mathsf{Q}$) is the problem with domain $\dom(\mathsf{P}) \times \dom(\mathsf{Q})$ and $(\mathsf{P} \times \mathsf{Q})(\seq{f,g}) = \mathsf{P}(f) \times \mathsf{Q}(g)$ for all $f \in \dom(\mathsf{P})$ and $g \in \dom(\mathsf{Q})$.
		\item $\mathsf{P} * \mathsf{Q}$ (the \emph{compositional product} of $\mathsf{P}$ and $\mathsf{Q}$) is the problem whose instances are all pairs $\seq{g,\Delta}$ such that $g \in \dom(\mathsf{Q})$ and $\Delta$ is a Turing functional with $\Delta(g,\hat{g}) \in \dom(\mathsf{P})$ for all $\hat{g} \in \mathsf{Q}(g)$, with the solutions to any such $\seq{g,\Delta}$ being all $\seq{\hat{f},\hat{g}}$ such that $\hat{g} \in \mathsf{Q}(g)$ and $\hat{f} \in \mathsf{P}(\Delta(g,\hat{g}))$.
		\item $\mathsf{P}^*$ (the \emph{finite parallelization} of $\mathsf{P}$) is the problem whose instances are all $\seq{k,\seq{f_0,\ldots,f_{k-1}}}$ where $k \geq 1$ and $f_0,\ldots,f_{k-1}$ are $\mathsf{P}$-instances, with $\mathsf{P}^*(\seq{k,\seq{f_0,\ldots,f_{k-1}}}) = \mathsf{P}(f_0) \times \cdots \times \mathsf{P}(f_{k-1})$.
		\item $\mathsf{P}'$ (the \emph{jump} of $\mathsf{P}$) is the problem whose instances are all pairs $\seq{f,i}$ such that $f \in \omega^\omega$ and $i$ is a $\Delta^{0,f}_2$ index for a $\mathsf{P}$-instance $g$ of $\mathsf{P}$, with a solution to any such $\seq{f,i}$ being all the $\mathsf{P}$-solutions to $g$. We write $\mathsf{P}^{(0)} = \mathsf{P}$, and for $n \in \omega$, $\mathsf{P}^{(n+1)} = (\mathsf{P}^{(n)})'$.
	\end{enumerate}	
\end{definition}

\noindent The Weihrauch degrees form a lattice with $\sqcup$ and $\sqcap$ as join and meet, respectively. We refer the reader to Brattka and Pauly \cite{BP-TA} for a comprehensive overview of the algebraic structure of the Weihrauch degrees under these (and many other) operations.

We will formally define the first-order part of a problem in the next section. For now, we make explicit the idea of a first-order, or ``number-theoretic'', problem.

\begin{definition}
	$\mathsf{P} \in \PROB$ is a \emph{first-order} problem if $\mathsf{P}(f) \subseteq \omega$ for all $f \in \dom(\mathsf{P})$. We let $\FPROB$ denote the class of all first-order problems.
\end{definition}

\noindent $\FPROB$ is a large class of problems with many important and ubiquitous members, e.g., $\mathsf{LPO}$, $\C_2$, $\lim_{\mathbb{N}}$, etc., which are commonly encountered in the literature on Weihrauch degrees. There are also many problems which, while not first-order themselves, are Weihrauch equivalent to first-order problems. The first-order part of a problem, which we define in the next section, will turn out to be the $\ured$-largest member of $\FPROB$ that can be Weihrauch reduced to it.

One aspect of our interest is in how the first-order part of a problem in the present setting compares with its first-order part as a theorem of second-order arithmetic (in cases where both perspectives make sense). In reverse mathematics, first-order parts are often measured against induction and bounding schemes, which are themselves stratified by the arithmetical hierarchy. The commonly used base theory in reverse mathematics, $\RCA_0$, includes induction for $\Sigma^0_1$ formulas, $\mathsf{I}\Sigma^0_1$, and above this lies the \emph{Kirby-Paris hierarchy} of successively stronger schemes: $\mathsf{B}\Sigma^0_2 < \mathsf{I}\Sigma^0_2 < \mathsf{B}\Sigma^0_3 < \mathsf{I}\Sigma^0_3 < \cdots$. (See \cite[Sections 6.1--6.3]{DM-2022} for definitions and further details.) There is a natural correspondence between these schemes and certain basic problems from computable analysis. We recall their definitions. Here and below, we identify each $k \in \omega$ with the set $\set{i \in \omega : i < k}$ for notational convenience.

\begin{definition}
	Fix $k \in \omega \cup \set{\omega}$. The \emph{choice problem on $k$}, $\C_k$, is the problem whose instances are all enumerations of proper subsets of $k$, with the solutions being all $i \in k$ that are not enumerated.
\end{definition}

\noindent For convenience, we will usually think of the instances of $\mathsf{C}_k$, more explicitly, as functions $v : \omega \to k + 1$ such that: $k \nsubseteq \ran(v)$; if $v(s) \neq v(t)$ for some $s < t$ then $v(t) \neq k$; and if $v(s) = v(t)$ for some $s < t$ then $v(s) = v(u) = v(t)$ for all $s < u < t$. In this way, $\set{ i < k : (\exists t \leq s)[v(t) = i]}$ indicates the set of $i$ enumerated at or before stage $s$. It is customary to write $\mathsf{C}_{\mathbb{N}}$ in place of $\mathsf{C}_\omega$.

As noted by Brattka, Gherardi, and Pauly \cite[Section 9.3]{BGM-2012}, $(\mathsf{C}_{\mathbb{N}})^{(n-1)}$ corresponds to induction for $\Sigma^0_n$ formulas, $\mathsf{I}\Sigma^0_n$, while $(\mathsf{C}_2^*)^{(n-1)}$ corresponds to bounding for $\Sigma^0_n$ formulas, $\mathsf{B}\Sigma^0_n$. As we will see, this correspondence extends to first-order parts in many, but not all, examples. Notice that for any $k \in \omega \cup \set{\omega}$, $\mathsf{C}_k$ belongs to $\FPROB$, as does any combinations of $\mathsf{C}_k$ under finite parallelizations and any number of jumps.

In the context of classical reverse mathematics, we think of mathematical principles as ``trivial'' if they are provable in the base theory, typically $\RCA_0$. The analogous notion for problems under Weihrauch reducibility is that of being \emph{uniformly computable true}. Let $\Id : \omega^\omega \to \omega^\omega$ be the identity problem, i.e., $\Id(f) = \set{f}$ for all $f \in \omega^\omega$.



\begin{definition}
	Fix $\mathsf{P} \in \PROB$.
	\begin{enumerate}
		\item $\mathsf{P}$ is \emph{computably true} if $\mathsf{P} \cred \Id$, i.e., if every $\mathsf{P}$-instance $f$ has an $f$-computable solution.
		\item $\mathsf{P}$ is \emph{uniformly computably true} if $\mathsf{P} \ured \Id$, i.e., if there exists a Turing functional $\Gamma$ such that $\Gamma(f) \in \mathsf{P}(f)$ for every $f \in \dom(\mathsf{P})$.	
	\end{enumerate}	
\end{definition}

\noindent In \Cref{sec:UNIF}, we will explore problems whose first-order parts are trivial, and identify an even stronger property than being uniformly computably true that arguably aligns more closely with the reverse mathematics notion of triviality.

We end this section with one simple yet at first glance somewhat surprising application of isolating the class $\FPROB$, which is otherwise unrelated to our discussion. This is that $\FPROB$ can be used to characterize computably true problems.

\begin{theorem}\label{P:comptruechar}
	A problem $\mathsf{P} \in \PROB$ is computably true if and only if there exists a $\mathsf{Q} \in \FPROB$ such that $\mathsf{P} \ured \mathsf{Q}$.
\end{theorem}

\begin{proof}
	If $\mathsf{P} \ured \mathsf{Q}$ for some $\mathsf{Q} \in \FPROB$ then it is clear that $\mathsf{P}$ is computably true. In the opposite direction, fix a computably true $\mathsf{P}$. Let $\mathsf{Q}$ be the problem with the same instances as $\mathsf{P}$, with the solutions to an instance $f$ being all $e \in \omega$ such that $\Phi_e(f) \in \mathsf{P}(f)$. Then $\mathsf{Q} \in \FPROB$ and $\mathsf{P} \ured \mathsf{Q}$.
\end{proof}

\section{Defining the first-order part}\label{sec:defining}

Restating the motivation from the preceding section, we would like the first-order part of a problem to correspond to the the strongest first-order problem that Weihrauch reduces to it. The definition we now give does not resemble this, but we will prove that it captures the same idea.

\begin{definition}\label{def:MAIN_def}
	For $\mathsf{P} \in \PROB$, the \emph{first-order part of $\mathsf{P}$}, denoted by $\fopart\mathsf{P}$, is the following first-order problem:
	\begin{itemize}
		\item the $\fopart\mathsf{P}$-instances are all triples $\seq{f,\Phi,\Psi}$, where $f \in \omega^\omega$ and $\Phi$ and $\Psi$ are Turing functionals such that $\Phi(f) \in \dom(\mathsf{P})$ and $\Psi(f,g)(0) \downarrow$ for all $g \in \mathsf{P}(\Phi(f))$;
		\item the $\fopart\mathsf{P}$-solutions to any such $\seq{f,\Phi,\Psi}$ are all $y$ such that $\Psi(f,g)(0) \downarrow = y$ for some $g \in \mathsf{P}(\Phi(f))$.
	\end{itemize}
\end{definition}

We proceed to our main theorem.

\begin{theorem}\label{thm:mainEQUIV}
	Fix $\mathsf{P} \in \mathcal{P}$.
	\begin{enumerate}
		\item For every $\mathsf{Q} \in \FPROB$, if $\mathsf{Q} \ured \mathsf{P}$ then $\mathsf{Q} \sured \fopart \mathsf{P}$.
		\item $\fopart\mathsf{P} \uequiv \max_{\ured} \{ \mathsf{Q} \in \FPROB : \mathsf{Q} \ured \mathsf{P}\}$.
	\end{enumerate}
\end{theorem}

\begin{proof}
	For part (1), suppose $\mathsf{Q} \ured \mathsf{P}$ for some $\mathsf{Q} \in \FPROB$, say via $\Phi$ and $\Psi$. We show that $\mathsf{Q} \sured \fopart\mathsf{P}$. Map a given instance $f$ of $\mathsf{Q}$ to $\seq{f,\Phi,\Psi}$. By assumption, $\Phi(f)$ is an instance of $\mathsf{P}$ and if $g$ is any $\mathsf{P}$-solution to $\Phi(f)$ then $\Psi(f, g)$ is a $\mathsf{Q}$-solution to $f$. As $\mathsf{Q}$ is first-order, this solution is simply $\Psi(f,g)(0)$. In particular, the latter converges, so $\seq{f,\Phi,\Psi}$ is an instance of $\fopart \mathsf{P}$. Now if $y$ is any $\fopart \mathsf{P}$-solution to $\seq{f,\Phi,\Psi}$ then by definition $y = \Psi(f,g)(0)$ for some $\mathsf{P}$-solution $g$ to $\Phi(f)$, so $y$ is also a $\mathsf{Q}$-solution to $f$.

	In light of part (1), to prove part (2) it suffices to show that $\fopart \mathsf{P} \ured \mathsf{P}$. To see this, we map a given $\fopart\mathsf{P}$-instance $\seq{f,\Phi,\Psi}$ to the $\mathsf{P}$-instance $\Phi(f)$. Then, we can map each $\mathsf{P}$-solution $g$ to $\Phi(f)$ to the output of the calculation $\Psi(f, g)(0)$. By definition, the latter is a $\fopart\mathsf{P}$-solution to $\seq{f,\Phi,\Psi}$.
\end{proof}

The following are immediate consequences of \Cref{def:MAIN_def} and the theorem.

\begin{corollary}\label{c:FPROBequalto1}
	If $\mathsf{P}, \mathsf{Q} \in \PROB$ and $\mathsf{Q} \ured \mathsf{P}$ then $\fopart \mathsf{Q} \sured \fopart \mathsf{P}$. In particular, if $\mathsf{Q} \uequiv \mathsf{P}$ then $\fopart \mathsf{Q} \suequiv \fopart \mathsf{P}$.
\end{corollary}


\begin{corollary}\label{c:FPROBequalto2}
	If $\mathsf{P} \in \PROB$ and $\mathsf{P} \uequiv \mathsf{Q}$ for some $\mathsf{Q} \in \FPROB$ then $\fopart \mathsf{P} \uequiv \mathsf{P}$.
\end{corollary}

The behavior of the first-order part under the standard operations on the Weihrauch degrees was studied in detail by Sold\`{a} and Valenti \cite{SV-TA}. In particular, they established the following basic bounds.

\begin{proposition}[Sold\`{a} and Valenti \cite{SV-TA}, Propositions 4.1 and 4.4]\label{SV_character}
	Fix $\mathsf{P},\mathsf{Q} \in \PROB$.
	\begin{enumerate}
		\item $\fopart (\mathsf{P} \sqcup \mathsf{Q}) \uequiv \fopart \mathsf{P} \sqcup \fopart \mathsf{Q}$.
		\item  $\fopart (\mathsf{P} \sqcap \mathsf{Q}) \uequiv \fopart \mathsf{P} \sqcap \fopart \mathsf{Q}$.
		\item $\fopart \mathsf{P} \times \fopart \mathsf{Q} \ured \fopart (\mathsf{P} \times \mathsf{Q})$.
		\item $\fopart \mathsf{P} * \fopart \mathsf{Q} \ured \fopart (\mathsf{P} * \mathsf{Q}) \ured \fopart \mathsf{P} * \mathsf{Q}$.
		\item $\fopart (\mathsf{P}') \sured (\fopart \mathsf{P})'$. 
	\end{enumerate}
	No additional relations hold in general.
\end{proposition}

Let us now pass to one specific example. Recall that for $k \geq 1$, $\RT^1_k$ denotes the problem whose instances are all functions $c : \omega \to k$ (called \emph{$k$-colorings} or just \emph{colorings}), with the solutions to any such $c$ being all its infinite \emph{monochromatic} sets, i.e., infinite sets $H \subseteq \omega$ on which $c$ is constant. (This is \emph{Ramsey's theorem} for $k$-colorings of singletons. We will discuss Ramsey's theorem in more generality in \Cref{sec:FOPARTS}.) There is a variant of this problem denoted $\mathsf{BWT}_k$, introduced by Brattka, Gherardi, and Marcone \cite{BGM-2012}. 
This has the same instances as those of $\RT^1_k$, but the solutions to any $c : \omega \to k$ are all $i < k$ such that $c(x) = i$ for infinitely many $x$. Now, even though $\mathsf{BWT}_k$ is first-order and $\RT^1_k$ is not, it is easy to see that $\RT^1_k \uequiv \mathsf{BWT}_k$. (See, e.g., \cite{BR-2017}, Proposition 3.4, for a proof.)  Hence, by \Cref{c:FPROBequalto1,c:FPROBequalto2} we have that $\fopart \RT^1_k \suequiv \fopart \mathsf{BWT}_k$ and $\fopart \RT^1_k \uequiv \RT^1_k$.

We can characterize the first-order part of $\RT^1_k$ in terms of more basic problems and operations from computable analysis. (This is our first example of such a characterization, but we will see others in the next two sections.) The following well-known result is due to Brattka, Gherardi, and Marcone \cite{BGM-2012}. For completeness, we include a proof here, which is also a bit more direct.

\begin{proposition}[Brattka, Gherardi, and Marcone \cite{BGM-2012}, Corollary 11.11]\label{fopart_rt1}
	For all $k \geq 1$, $\mathsf{BWT}_k \suequiv \mathsf{C}'_k$.
\end{proposition}

\begin{proof}
	First, fix an instance of $\C_k'$. Regard this as an $f \in \omega^\omega$ and a $\Delta^{0,f}_2$ index for an $f'$-computable instance $v$ of $\C_k$. The nonempty set $S = \set{i < k : (\forall s)[ v(s) \neq i]}$ is then uniformly $\Pi^0_2$ in $f$. Let $R$ be an $f$-computable predicate so that $i \in S$ if and only if $(\forall u)(\exists v)R(i,u,v)$. Define an $f$-computable coloring $c : \omega \to k$ as follows. Given $x \in \omega$, search for the least $y$ such that $(\exists i < k)(\forall u < x)(\exists v < y)R(i,u,v)$, and let $c(x)$ be the least witness $i$ for this $y$. Now suppose $i < k$ is a $\mathsf{BWT}_k$-solution to $c$. Then in particular there are infinitely many $x$ such that $(\forall u < x)(\exists v)R(i,u,v)$ and so $i \in S$. It follows that $i$ is a $\C_2$-solution to $v$, as wanted.
	
	Conversely, fix an instance $c$ of $\mathsf{BWT}_k$. Define $v : \omega \to k+1$ as follows: let $v(0) = k$, and for all $s > 0$, let $v(s) = i$ for the least $i < k$ such that $(\forall x > s)[c(x) \neq i]$ and $v(t) \neq i$ for any $t < s$, or $v(s-1)$ if no such $i$ exists. Then $v$ is a uniformly $c'$-computable instance of $\C_k$, so we can regard $c$ together with a $\Delta^{0,c}_2$ index for $v$ as an instance of $\C_k'$. Any $\C_k$-solution $i$ to $v$ is a $\mathsf{BWT}_k$-solution to $c$.
\end{proof}

\begin{corollary}
	For all $k \geq 1$, $\fopart \RT^1_k \suequiv \fopart \mathsf{BWT}_k \suequiv \C_k'$.	
\end{corollary}

\noindent For $k \geq 2$, one takeaway here is that while $\RT^1_k$ has trivial first-order part as a $\Pi^1_2$ statement of second-order arithmetic, its first-order part in the present setting is non-trivial. Of course, $\RT^1_k$ is \emph{itself} trivial as a $\Pi^1_2$ statement, and nontrivial as a problem, so this is not so surprising. In the next two sections, we will see some more interesting examples of this phenomenon.

A further insight we can obtain from $\RT^1_k$ is that, in general, it is false that $\fopart \mathsf{P} \sured \mathsf{P}$, even for $\mathsf{P} \in \FPROB$. (Thus, \Cref{c:FPROBequalto2} cannot be improved to $\suequiv$, even if $\mathsf{P} \suequiv \mathsf{Q}$ there.)

\begin{proposition}
	For all $k \geq 2$, $\fopart \mathsf{BWT}_k \nsured \mathsf{BWT}_k$ and $\fopart \mathsf{BWT}_k \nsured \mathsf{RT}^1_k$.
\end{proposition}

\begin{proof}
	Fix $k \geq 2$. To show that $\fopart \mathsf{BWT}_k \nsured \mathsf{BWT}_k$, let $\Phi$ be the identity functional on $\omega^\omega$. Let $\Psi$ be a functional such that $\Psi(f,i)(0)$ is the least $x \in \omega$ such that $f(x) = i$, for all $f \in \omega^\omega$ and $i \in \omega$. So if $c$ is an instance of $\mathsf{BWT}_k$ and $i$ is a solution to $c$ then $\Psi(c,i)(0) \downarrow$. It follows that $\seq{c,\Phi,\Psi}$ is an instance of $\fopart \mathsf{BWT}_k$ for every such $c$. Now, suppose towards a contradiction that $\fopart \mathsf{BWT}_k \sured \mathsf{BWT}_k$, say via $\widehat{\Phi}$ and $\widehat{\Psi}$. For each $x \in \omega$, let $c_x = 0^x 1^\omega$, viewed as a coloring $\omega \to k$. Then for each $x \leq k$, we have by assumption that $\widehat{\Phi}(c_x)$ is a $\mathsf{BWT}_k$-instance. There must then exist $x_0 < x_1 \leq k$ and $i < k$ such that $\widehat{\Phi}(c_{x_0})$ and $\widehat{\Phi}(c_{x_1})$ each have $i$ as a $\mathsf{BWT}_k$-solution. Hence, $\widehat{\Psi}(i)$ must be a $\fopart \mathsf{BWT}_k$-solution to both $\seq{c_{x_0},\Phi,\Psi}$ and $\seq{c_{x_1},\Phi,\Psi}$. But since each of $c_{x_0}$ and $c_{x_1}$ has a unique $\mathsf{BWT}_k$-solution of $1$, it follows that $\seq{c_{x_0},\Phi,\Psi}$ and $\seq{c_{x_1},\Phi,\Psi}$ have unique $\fopart \mathsf{BWT}_k$-solutions $x_0$ and $x_1$, respectively, and $x_0 \neq x_1$. A similar (but simpler) argument shows that $\fopart \mathsf{BWT}_k \nsured \mathsf{RT}^1_k$.
\end{proof}

\noindent Note that we do need both parts above because for $k \geq 2$, $\mathsf{BWT}_k$ and $\RT^1_k$ are incomparable under strong Weihrauch reducibility. We are not aware of any explicit proof of this in the literature, but it is straightforward and routine. Trivially, $\fopart \mathsf{BWT}_1 \suequiv \mathsf{BWT}_1 \suequiv 
\RT^1_1$.

\section{Uniform computable solvability and undiagonalizability}\label{sec:UNIF}

In this section, we explore a bit more the notion of being uniformly computable true (i.e., trivial under Weihrauch reducibility) and how it interacts with the first-order part of a problem. To begin, we connect this notion with the following one, which was introduced by Hirschfeldt and Jockusch \cite{HJ-2016} in an unrelated context.

\begin{definition}[Hirschfeldt and Jockusch \cite{HJ-2016}, Definition 4.11]
	A problem $\mathsf{P}$ is \emph{undiagonalizable} if for every $\mathsf{P}$-instance $f$, the set of $\sigma \in \omega^{<\omega}$ that can be extended to a $\mathsf{P}$-solution $g \in \omega^\omega$ to $f$ is uniformly $\Delta^0_1$ in $f$ (i.e., there is a Turing functional $\Gamma$ such that $\Gamma(f)(\sigma) \downarrow = i \in \{0,1\}$ for all $\mathsf{P}$-instances $f$ and all $\sigma \in \omega^{<\omega}$, with $i = 1$ if and only if $\sigma$ is an initial segment of $\mathsf{P}$-solution to $f$.)
\end{definition}

\noindent Notice that any problem can be made undiagonalizable without changing the Turing degrees of its solutions, simply by replacing each solution by all finite modifications of it. In particular, there are many examples of such problems that are not themselves uniformly computably true, even non-uniformly so. This makes the next result striking.

\begin{proposition}\label{P:undiag-to-trivial}
	Let $\mathsf{P} \in \PROB$ be undiagonalizable. Then $\fopart\mathsf{P}$ is uniformly computably true.
\end{proposition}

\begin{proof}
	Fix a functional $\Gamma$ witnessing that $\mathsf{P}$ is undiagonalizable. Given any instance $\seq{f,\Phi,\Psi}$ of $\fopart \mathsf{P}$, search for the first $\sigma \in \omega^{<\omega}$ such that $\Gamma(\Phi(f))(\sigma) \downarrow = 1$ and $\Psi(f,\sigma)(0) \downarrow$. (The search must succeed since any sufficiently long initial segment of any $\mathsf{P}$-solution to $\Phi(f)$ can serve as $\sigma$.) The value of $\Psi(f,\sigma)(0) \downarrow$ is then a $\fopart\mathsf{P}$-solution to $\seq{f,\Phi,\Psi}$.
\end{proof}

\noindent As we will see, being undigonalizable somewhat better captures the idea of having trivial first-order part than simply having the first-order part be uniformly computably true.

The converse of \Cref{P:undiag-to-trivial} is false. In fact, being uniformly computably true does not even imply being Weihrauch reducible to an undiagonalizable problem. To see this, consider the \emph{thin set principle} for $3$-colorings of singletons, $\TS^1_3$. Its instances are all colorings $c : \omega \to 3$, and the solutions to any such $c$ are all its infinite \emph{thin} sets, i.e., infinite sets $T \subseteq \omega$ such that $|c(T)| \leq 2$. Hirschfeldt and Jockusch \cite[p.~39]{HJ-2016} point out that $\TS^1_3$ has what they call \emph{diagonalization opportunities} (see \cite[Definition 4.12]{HJ-2016}) and they show that no problem that has diagonalization opportunities is Weihrauch reducible to any undiagonalizable problem (\cite[Theorem 4.13]{HJ-2016}). But $\fopart \TS^1_3$ is uniformly computably true. Indeed, given an instance $\seq{f,\Phi,\Psi}$ of $\fopart \TS^1_3$ with $\Phi(f) = c : \omega \to 3$, we can search for $\sigma \in 2^{<\omega}$ such that $c \res \{x < |\sigma| : \sigma(x) = 1\}$ is constant and $\Psi(f,\sigma)(0) \downarrow$. Any monochromatic set for $c$ is also thin for $c$, and any finite $c$-homogeneous set is extendible to an infinite $c$-thin one. Thus, the search must succeed and $\Psi(f,\sigma)(0)$ must be a $\fopart \TS^1_3$-solution to $\seq{f,\Phi,\Psi}$.

We have the following immediate consequence of \Cref{P:undiag-to-trivial} and \Cref{c:FPROBequalto2}.


\begin{corollary}
	If $\mathsf{P} \in \PROB$ is undiagonalizable but not uniformly computably true then no $\mathsf{Q} \in \FPROB$ satisfies $\mathsf{Q} \uequiv \mathsf{P}$.
\end{corollary}

\noindent This has an interesting application in that we can use it to see that in \Cref{P:comptruechar}, we cannot in general replace $\ured$ by $\uequiv$. Indeed, consider any undiagonalizable problem that is computably true but not uniformly computably true. (For example, this can be the problem ${}^{\text{FE}}\RT^1_2$, introduced by Dzhafarov, Goh, Hirschfeldt, Patey, and Pauly  \cite{DGHPP-2020}, Definition 1.7. This is just $\RT^1_2$, but with solutions replaced by all infinite sets that are homogeneous modulo finitely many errors.) By \Cref{P:comptruechar}, there is a $\mathsf{Q} \in \FPROB$ such that $\mathsf{P} \ured \mathsf{Q}$, but by the preceding corollary, no such $\mathsf{Q}$ satisfies $\mathsf{P} \uequiv \mathsf{Q}$.


Let us next look at some better-known examples of undiagonalizable problems. We recall some definitions.

\begin{definition}
	\
	\begin{enumerate}
		\item A set $X \subseteq \omega$ is \emph{cohesive} for a family $\seq{A_i : i \in \omega}$ of subsets of $\omega$ if for all $i$, either $X \cap A_i$ or $X \cap \overline{A}_i$ is finite.
		\item A family of sets $\seq{B_i : i \in \omega}$ is a \emph{subfamily} of a family of sets $\seq{A_i : i \in \omega}$ if $(\forall i)(\exists j)[B_i = A_j]$. We write $\seq{B_i : i \in \omega} \subseteq \seq{A_i : i \in \omega}$.
		\item A family of sets $\seq{A_i : i \in \omega}$ has the \emph{finite intersection property} if $\bigcap_{i \in F} A_i \neq \emptyset$ for every nonempty finite set $F$.
	\end{enumerate}
\end{definition}

\noindent The following problems come from the reverse mathematics literature, but have also been studied to a lesser extend in the context of Weihrauch reducibility.

\begin{definition}
	\
	\begin{enumerate}
		\item $\COH$ is the problem whose instances are all families of sets $\seq{A_i : i \in \omega}$, with the solutions being all the infinite cohesive sets for this family.
		\item $\mathsf{FIP}$ is the problem whose instances are all families of sets $\seq{A_i : i \in \omega}$, not all empty, with the solutions being all the $\subseteq$-maximal subfamilies of $\seq{A_i : i \in \omega}$ that have the finite intersection property.
		\item $\Pi^0_1\mathsf{G}$ is the principle whose instances are all $f \in \omega^\omega$ and all $\Pi^{0,f}_1$-definable families $\seq{U_i : i \in \omega}$ of nonempty subsets of $2^{<\omega}$, with the solutions being  all sets $G \subseteq \omega$ that meet every $U_i$ (i.e., $(\forall i)(\exists k)[G \res k \in U_i]$).
	\end{enumerate}	
\end{definition}

\noindent (See, e.g., \cite{DM-2022}, Section 8.4.2 for a broader discussion of $\COH$, and Section 9.10.3 for a broader discussion of $\mathsf{FIP}$ and $\Pi^0_1\mathsf{G}$.) $\COH$ and $\Pi^0_1\mathsf{G}$ are undiagonalizable because every finite binary string can be continued to a solution of a given instance. Dzhafarov and Mummert \cite[Proposition 4.2]{DM-2013} showed that, as $\Pi^1_2$ principles, $\mathsf{FIP}$ is implied by $\Pi^0_1\mathsf{G}$ over $\RCA_0$, and their proof actually shows that $\mathsf{FIP} \ured \Pi^0_1\mathsf{G}$ as problems. $\mathsf{FIP}$ is not itself undiagonalizable, but it turns out to be undiagonalizable up to Weihrauch equivalence. Indeed, $\mathsf{FIP}$ clearly satisfies the weaker property in the hypothesis of the following result. 

\begin{proposition}
	Let $\mathsf{P}$ be a problem such that the set of $\sigma \in \omega^{<\omega}$ that can be extended to a $\mathsf{P}$-solution $g \in \omega^\omega$ to $f$ is uniformly $\Sigma^0_1$ in $f$. Then there is an undiagonalizable problem $\mathsf{Q}$ such that $\mathsf{P} \uequiv \mathsf{Q}$.
\end{proposition}

\begin{proof}
	Fix $\mathsf{P}$, and let $W_e$ be such that for every $\mathsf{P}$-instance $f$, $W_e^f$ is the set of all the initial segments of the $\mathsf{P}$-solutions to $f$. Let $\mathsf{Q}$ be the problem with the same instances as $\mathsf{P}$, but with the solutions to a $\mathsf{Q}$-instance $f$ being all sequences of the form $\seq{s_0,\sigma_0}\seq{s_1,\sigma_1}\cdots \in \omega^\omega$ such that $s_0 \leq s_1 \leq \cdots$, $\sigma_0 \prec \sigma_1 \prec \cdots$, $\sigma_k \in W_e^f[s_k]$ for every $k$, and $\bigcup_{k \in \omega} \sigma_k$ is a $\mathsf{P}$-solution to $f$. It is easy to see that $\mathsf{Q}$ is undiagonalizable and that $\mathsf{P} \uequiv \mathsf{Q}$.
\end{proof}

\noindent It follows that $\mathsf{FIP}$ is Weihrauch equivalent to an undiagonalizable problem. By \Cref{P:undiag-to-trivial}, we can now conclude the following.

\begin{corollary}\label{prop:FOPART_COH}
	Each of $\fopart \COH$, $\fopart \mathsf{FIP}$, and $\fopart \Pi^0_1\mathsf{G}$ is uniformly computably true. 
\end{corollary}

The corollary nicely meshes with what is known about the first-order parts of $\COH$, $\mathsf{FIP}$, and $\Pi^0_1\mathsf{G}$ as $\Pi^1_2$ statements of second-order arithmetic. Each of these principles is $\Pi^1_1$-conservative over $\RCA_0$. For $\COH$, this fact is due to Cholak, Jockusch, and Slaman \cite[Theorem 9.1]{CJS-2001}. For $\Pi^0_1\mathsf{G}$, it is due for Hirschfeldt, Shore, and Slaman \cite[Theorem 3.13 and the comment on p.~5824]{HSS-2009}. The latter also implies this fact for $\mathsf{FIP}$. Thus, in these cases, the first-order strengths agree between the classical reverse mathematics and Weihrauch analysis settings: they are trivial.

We wrap up this section by looking at how being uniformly computably true and being undiagonalizable behave under the standard operations on Weihrauch degrees.

\begin{proposition}
	Fix $\mathsf{P},\mathsf{Q} \in \PROB$ with $\fopart \mathsf{Q}$ uniformly computably true.
	\begin{enumerate}
		\item $\fopart (\mathsf{P} \sqcup \mathsf{Q}) \ured \fopart \mathsf{P}$.
		\item  $\fopart (\mathsf{P} \sqcap \mathsf{Q}) \ured \fopart \mathsf{P}$.
		\item $\fopart \mathsf{P} \times \fopart \mathsf{Q} \ured \fopart \mathsf{P}$.
		\item $\fopart \mathsf{P} * \fopart \mathsf{Q} \ured \fopart \mathsf{P}$.
		\item $\fopart \mathsf{Q} * \fopart \mathsf{P} \ured \fopart (\mathsf{Q} * \mathsf{P}) \ured \fopart \mathsf{P}$.
		
	\end{enumerate}
	If $\mathsf{Q}$ is undiagonalizable, then additionally $\fopart (\mathsf{P} \times \mathsf{Q}) \ured \fopart \mathsf{P}$ and $\fopart (\mathsf{P} * \mathsf{Q}) \ured \fopart \mathsf{P}$.
\end{proposition}

\begin{proof}
	The first four parts are straightforward, using \Cref{SV_character} in the case of parts (1) and (2). Also by \Cref{SV_character}, we have that $\fopart \mathsf{Q} * \fopart \mathsf{P} \ured \fopart (\mathsf{Q} * \mathsf{P}) \ured \fopart \mathsf{Q} * \mathsf{P}$. Since $\fopart \mathsf{Q}$ is uniformly computably true, $\fopart \mathsf{Q} * \mathsf{P} \ured \mathsf{P}$. Now (5) follows because $\fopart \mathsf{Q} * \fopart \mathsf{P}$ and $\fopart (\mathsf{Q} * \mathsf{P})$ are first-order.
	
	Now suppose $\mathsf{Q}$ is undiagonalizable. We reduce each of $\fopart (\mathsf{P} \times \mathsf{Q})$ and $\fopart (\mathsf{P} * \mathsf{Q})$ to $\mathsf{P}$, which suffices. First, fix an instance $\seq{f,\Phi,\Psi}$ of $\fopart (\mathsf{P} \times \mathsf{Q})$, so that $\Phi(f)$ is an instance $\seq{f_0,f_1}$ of $\mathsf{P} \times \mathsf{Q}$. We map this to the $\mathsf{P}$-instance $f_0$. Given any $\mathsf{P}$-solution $g$ to $f_0$, we search for an initial segment $\sigma \in \omega^{<\omega}$ of a $\mathsf{Q}$-solution to $f_1$ such that $\Psi(f,\seq{g,\sigma})(0) \downarrow$. Note that this search is uniformly computable in $f$ since $\mathsf{Q}$ is undiagonalizable, and it must succeed since any sufficiently long initial segment of any $\mathsf{Q}$-solution to $f_1$ would work. The value of $\Psi(f,\seq{g,\sigma})(0)$ is then a $\fopart (\mathsf{P} \times \mathsf{Q})$-solution to $\seq{f,\Phi,\Psi}$. The argument for $\fopart (\mathsf{P} * \mathsf{Q})$ is similar.
\end{proof}


\section{Additional case studies}\label{sec:FOPARTS}

We have already classified the first-order parts of several problems whose first-order parts as theorems of second-order arithmetic were previously known. In this section, we look at several more examples. We begin with \emph{weak K\"{o}nig's lemma}, $\WKL$. As a theorem, $\WKL$ is famously $\Pi^1_1$-conservative over $\RCA_0$. (This is Harrington's theorem; see \cite[Section 7.2]{Hirschfeldt-2014} or \cite[Section 7.7]{DM-2022}.) Unlike $\COH$, $\mathsf{FIP}$, and $\Pi^0_1\mathsf{G}$ from \Cref{sec:UNIF}, $\WKL$ is not undiagonalizable, so it does not follow that its first-order part as a problem is also trivial, and indeed this turns out not to be the case. In the following theorem, we recall that \emph{weak weak K\"{o}nig's lemma}, $\WWKL$, is $\WKL$ with instances restricted to trees $T \subseteq 2^{<\omega}$ so that $[T] \subseteq 2^\omega$ has positive Lebesgue measure.

\begin{theorem}\label{fopart_wkl}
	For all $n \in \omega$, we have
	\[
		\fopart (\WWKL^{(n)}) \suequiv \fopart (\WKL^{(n)}) \suequiv (\mathsf{C}^*_2)^{(n)}.
	\]
\end{theorem}

\begin{proof}
	Fix $n$. We first show that $(\mathsf{C}^*_2)^{(n)} \sured \fopart (\WWKL^{(n)})$. By \Cref{thm:mainEQUIV}, since $(\mathsf{C}^*_2)^{(n)} \in \FPROB$, we can just show that $(\mathsf{C}^*_2)^{(n)} \sured \WWKL^{(n)}$. And since the jump operator is monotone on $\sured$, it suffices to show that $\mathsf{C}^*_2 \sured \WWKL$. To this end, let $\seq{v_0,\ldots,v_{k-1}}$ be any instance of $\C_2^*$. Define $T$ to be the set of all $\sigma \in 2^{<\omega}$ such that for all $e < |\sigma|$,
	\[
	\sigma(e) =
	\begin{cases}
		0 & \text{if } e < k,\\
		1 & \text{if } e = k,\\
		1-v_{e-k-1}(|\sigma|) & \text{if } k < e \leq 2k \text{ and } v_{e-k-1}(|\sigma|) < 2.
	\end{cases}
	\]
	Note that if $\sigma \in T$ and $v_{e-k-1}(|\tau|) \neq v_{e-k-1}(|\sigma|)$ for some $k < e \leq 2k$ and some $\tau \preceq \sigma$ then necessarily $v_{e-k-1}(|\tau|) = 2$, so trivially $\tau \in T$. Hence, $T$ is a tree. Furthermore, by construction, the elements of $[T]$ are precisely the sequences of the form $0^{k}1x_0 \cdots x_{k-1}g$, where $\seq{x_0,\ldots,x_{k-1}}$ is a $\C^*_2$-solution to $\seq{v_0,\ldots,v_{k-1}}$ and $g \in 2^\omega$ is arbitrary. It follows that the measure of $[T]$ is at least $2^{-2k-1}$ and so $T$ is an instance of $\WWKL$. Now if $p$ is any $\WWKL$-solution to $T$ then $k$ can be computably recovered as the least $e$ such that $p(e) = 1$, and then $\seq{p(k+e+1) : e < k}$ is a $\C^*_2$-solution to $\seq{v_0,\ldots,v_{k-1}}$.
	
	That $\fopart (\WWKL^{(n)}) \sured \fopart (\WKL^{(n)})$ is clear. It therefore remains only to show that $\fopart (\WKL^{(n)}) \sured (\mathsf{C}^*_2)^{(n)}$. By \Cref{SV_character}, $\fopart (\WKL^{(n)}) \sured (\fopart \WKL)^{(n)}$, so by the monotonicity of the jump operator on $\sured$, it suffices to show that $\fopart \WKL \sured \mathsf{C}_2^*$. We will work with the problem $\DNR_2$ instead of $\WKL$,
	whose instances are all $g \in 2^\omega$, with the solutions to any such $g$ being all $\{0,1\}$-valued functions that are diagonally noncomputable relative to $g$ (hereafter abbreviated DNC$^g$). By results of Brattka, Hendtlass, and Kreuzer~\cite[Corollary 5.3]{BHK-2017}, $\WKL \suequiv \DNR_2$
	
	Consider an instance of $\fopart \DNR_2$. Since the instances of $\DNR_2$ range over all elements of $\omega^\omega$, we may regard this simply as a pair $\seq{g,\Psi}$ where $g \in \omega^\omega$ and $\Psi(g,p)(0) \downarrow$ for every $\set{0,1}$-valued DNC$^g$ function $p$. Here, $\Psi$ is given by an index, $i$. Let $T_0 \subseteq 2^{<\omega}$ be the standard $g$-computable tree whose paths are precisely the $\{0,1\}$-valued DNC$^g$ functions. Let $T \subseteq 2^{<\omega}$ be the tree of all $\sigma$ of the form $0^i 1 \rho$ for $\rho \in T_0$.
	
	
	
	By compactness, there is a $k$ such that $\Psi(g,\sigma)(0) \downarrow$ for every $\sigma \in T_0$ of length $k$. We have to encode $i$, $k$, and $g \res k$ into the solutions of an instance of $\C_2^*$. For each $e < k$, define $w_e : \omega \to 3$ by
	\[
		w_e(s) =
		\begin{cases}
			2 & \text{if } \Phi_{e}(g)(e)[s] \uparrow,\\
			\min \set{\Phi_{e}(g)(e)[s],1}  & \text{otherwise},
		\end{cases}
	\]
	for all $s$. Thus, each $w_e$ is an instance of $\C_2$, and if $1 - \lim_s w_e(s) \neq \Phi_{e}(g)(e)$ if the latter converges.
	
	Now consider the sequence
	\[
		1^i{}^\frown{}0^\frown{}\seq{1-g(e) : e < k}^\frown{}\seq{w_e : e < k},
	\]
	where we regard each of $0$, $1$, and $1-g(e)$ as a constant function $\omega \to 3$. So we have an instance $\vec{v} = \seq{v_j : j < i+1+2k}$ of $\C_2^*$. Since $k$ and each $w_e$ can be uniformly computed from our $\fopart \DNR_2$-instance $\seq{g,i}$, it follows that we can uniformly compute $\vec{v}$ from this data. This is then our desired instance of $\mathsf{C}_2^*$.
	
	Now let $\seq{x_j : j < i+1+2k}$ be any solution to this instance. Since $x_j \neq \lim_s v_j(s)$, we can computably recover $i$ as the least $j$ such that $x_j = 1$. Using $i$ and the length of the solution, we can next also recover $k$ and $g \res k$. Finally, as remarked above, for $i+1+k \leq j < i+1+2k$ we must have $x_j \neq \Phi_{j-i-1-k}(g)(j-i-1-k)$. In other words, the string $\sigma \in 2^k$ defined by $\sigma(e) = x_{i+1+k+e}$ for all $e < k$ belongs to $T_0$ and so $\Psi(g \res k,\sigma)(0) \downarrow$ by choice of $k$. By assumption on $\Psi$, the value of this computation is a $\fopart \DNR_2$-solution to the instance we started with. Since we have shown that we can uniformly computably obtain this value from $\seq{x_j : j < i+1+2k}$, the proof is complete.
\end{proof}

%
%

The preceding result refines the $\Pi^1_1$-conservation of $\WKL$ mentioned above in an interesting way. 
Namely, it is known that $\WKL$ is $\Pi^1_1$-conservative not only over $\RCA_0$, but also over $\RCA_0^* + \mathsf{I}\Sigma^0_n$ and $\RCA_0^* + \mathsf{B}\Sigma^0_n$, for all $n \geq 1$. (See H\'{a}jek \cite[Corollary 3.14]{Hajek-1993} and Simpson and Smith \cite[Corollary 4.7]{SS-1986}.) More recently, Fiori-Carones, Ko\l{}odziejczyk, Wong, and Yokoyama \cite[Lemma 4.5]{Keita-NEW} defined a version of $\WKL$ for $\Delta^0_n$-definable trees, and proved that this is also $\Pi^1_1$-conservative over $\RCA_0^* + \mathsf{B}\Sigma^0_n$. Recall that in that Weihrauch degrees, $(\mathsf{C}_2^*)^{(n)}$ corresponds to $\mathsf{B}\Sigma^0_n$, so our theorem above is precisely analogous to the latter result.

We next turn to the \emph{arithmetical comprehension axiom}, $\ACA$, from reverse mathematics. Formally, this is $\RCA_0$ plus comprehension for all arithmetically-definable subsets of numbers. Often, it is presented in the form $(\forall X)[X' \text{ exists}]$, which relies on a formalization of computability theory in the base theory $\RCA_0$. (See \cite[Corollary 5.6.3]{DM-2022}.) This in turn has an obvious problem form as the \emph{Turing jump problem}, $\mathsf{TJ}$, whose instances are all $g \in \omega^\omega$, with the unique solution to any such $g$ being $g'$. But $\ACA_0$ is also equivalent to the statement $(\forall X)[X^{(n+1)} \text{ exists}]$, for any $n \in \omega$, so it could just as well be represented by $\mathsf{TJ}^{(n)}$ or even $\bigsqcup_{n \in \omega} \mathsf{TJ}^{(n)}$.

\begin{theorem}\label{ACA_fopart}
	For all $n \in \omega$,
	\[
		\fopart (\mathsf{TJ}^{(n)}) \suequiv \mathsf{C}_{\mathbb{N}}^{(n)}.
	\]	
\end{theorem}

\begin{proof}
	First, we show that $\fopart (\mathsf{TJ}^{(n)}) \sured \mathsf{C}_{\mathbb{N}}^{(n)}$. By \Cref{SV_character} and the monotonicity of the jump operator on $\sured$, it suffices to show that $\fopart \mathsf{TJ} \sured \mathsf{C}_{\mathbb{N}}$. So fix an instance of $\fopart \mathsf{TJ}$, which we may just think of as a pair $\seq{g,\Psi}$ where $g \in \omega^\omega$ and $\Psi(g,g')(0) \downarrow$. Let $m : \omega \to \omega$ be a $g$-computable limit approximation to the value of $\Psi(g,g')(0)$. Define $v : \omega \to \omega$ as follows. For each $s$, find the least $k$ such that $\seq{k,m(s)} \notin \ran(v \res s)$, and then let $v(s)$ be the least $ x$ different from $\seq{k,m(s)}$ and not in the range of $v \res s$. Note that since $y = \lim_s m(s)$ exists, there is a $k$ such that $\seq{k,y} \notin \ran(v)$. By construction, every $x \neq \seq{k,y}$ does belong to $\ran(v)$. Thus, $v$ is an instance of $\mathsf{C}_{\mathbb{N}}$ with $\seq{k,y}$ as its only solution. Moreover, $v$ is uniformly computable from $g$ because $m$ is, as desired.

	 
	 By \Cref{thm:mainEQUIV}, since $\mathsf{C}_{\mathbb{N}}^{(n)} \in \FPROB$, to show that $\mathsf{C}_{\mathbb{N}}^{(n)} \sured \fopart(\mathsf{TJ}^{(n)})$ it suffices to show that $\mathsf{C}^{(n)}_{\mathbb{N}} \sured \mathsf{TJ}^{(n)}$. By the monotonicity of the jump operator on $\sured$, for this it in turn suffices to show that $\mathsf{C}_{\mathbb{N}} \sured \mathsf{TJ}$. This is straightforward.
\end{proof}

\noindent Note that what the above actually shows is that $\mathsf{C}_{\mathbb{N}}^{(n)} \sured \mathsf{TJ}^{(n)} \sured \mathsf{UC}_{\mathbb{N}}^{(n)}$, where $\mathsf{UC}_{\mathbb{N}}^{(n)}$ is the \emph{unique choice problem on $\mathbb{N}$}, or $\mathsf{C}_{\mathbb{N}}$ restricted to instances with unique solutions. For completeness we note that the fact that $\mathsf{C}_{\mathbb{N}} \suequiv \mathsf{UC}_{\mathbb{N}}$ is well-known; see, e.g., Brattka, Gherardi, and Pauly \cite[Theorem 7.13]{BGP-TA}.

\begin{corollary}\label{cor:ACA_fopart}
	$\fopart \bigsqcup_{n \in \omega} \mathsf{TJ}^{(n)} \suequiv \bigsqcup_{n \in \omega} \mathsf{C}^{(n)}_{\mathbb{N}}$.	
\end{corollary}

As is well-known, $\ACA_0$ is $\Pi^1_1$-conservative over Peano arithmetic, $\PA$. (See, e.g., \cite[Corollary 7.5]{Hirschfeldt-2014} for a proof.) Effectively, this means that the first-order strength of $\ACA_0$ is arithmetical induction. \Cref{cor:ACA_fopart} bears this out very directly, while \Cref{ACA_fopart} can then be seen as a stratification of this result that is impossible to extract in the classical reverse mathematics setting.

For our final case study, we look at \emph{Ramsey's theorem}, which has been the subject of much study in reverse mathematics and computable analysis. (A detailed overview can be found in \cite[Chapter 8 and Section 9.1]{DM-2022}.)

\begin{definition}
	Fix $X \subseteq \omega$ and $n,k \geq 1$.
	\begin{enumerate}
		\item $[X]^n$ denotes the set of all $\seq{x_0,\ldots,x_{n-1}} \in X^n$ with $x_0 < \cdots < x_{n-1}$.
		\item A \emph{$k$-coloring} (or \emph{coloring} for short) of $[X]^n$ is a map $c : [X]^n \to k$.
		\item A $k$-coloring $c : [X]^n \to k$ is \emph{stable} if for all $x \in [X]^{n-1}$, $\lim_s c(\vec{x},s)$ exists.
		\item A set $Y \subseteq X$ is \emph{homogeneous} for $c : [X]^n \to k$ if $c \res\, [Y]^n$ is constant.
		\item $\RT^n_k$ is the problem whose instances are all colorings $c : [\omega]^n \to k$, with the solutions to any such $c$ being all its infinite homogeneous sets.
		\item $\SRT^n_k$ is the restriction of $\RT^n_k$ to stable colorings.
		\item $\RT^n_{+} = \bigsqcup_{k \geq 2} \RT^n_k$ and $\SRT^n_{+} = \bigsqcup_{k \geq 2} \SRT^n_k$.
		\item $\RT^n_{\mathbb{N}} = \bigcup_{k \geq 1} \RT^n_k$ and $\SRT^n_{\mathbb{N}} = \bigcup_{k \geq 1} \SRT^n_k$.
	\end{enumerate}	
\end{definition}

\noindent The variants $\RT^n_{+}$ and $\RT^n_{\mathbb{N}}$ were introduced by Brattka and Rakotoniaina \cite[Definition 3.1]{BR-2017}. In both, the instances are $k$-colorings of exponent $n$ for \emph{some} $k$, with the difference being merely that in $\RT^n_{+}$, this $k$ is specified as part of the instance, whereas in $\RT^n_{\mathbb{N}}$ it is not. Thus, $\RT^n_{+} \ured \RT^n_{\mathbb{N}}$, and Brattka and Rakotoniaina \cite[Corollary 4.23]{BR-2017} proved that $\RT^n_{\mathbb{N}} \nured \RT^n_{+}$. But both problems correspond to one and the same $\Pi^1_2$ statement of second-order arithmetic, namely $(\forall k)\RT^n_k$, which is denoted in the reverse mathematics literature by $\RT^n$ or $\RT^n_{<\infty}$. Analogously for the stable variants, $\SRT^n_{+}$ and $\SRT^n_{\mathbb{N}}$.

The upper bounds in the following theorem were obtained independently by Sold\`{a} and Valenti \cite[Section 7.1]{SV-TA}, who were looking instead at the problems $\RT^n_k$ and $\SRT^n_k$ for finite values of $k$. The upper bound in the case $n = 0$ was also obtained, by different means, by Brattka and Rakotoniaina \cite{BR-2017}.

\begin{theorem}\label{FOPART_RAMSEY}
	For all $n \geq 1$, we have
	\[
		(\mathsf{C}_2^{(n)})^* \ured \fopart \SRT^n_{\mathbb{N}} \ured \fopart \RT^n_{\mathbb{N}} \ured (\mathsf{C}^{*}_2)^{(n)}
	\]
	and
	\[
		(\mathsf{C}_2^{(n)})^* \ured \fopart \SRT^n_{+} \ured \fopart \RT^n_{+} \ured (\mathsf{C}^{*}_2)^{(n)}.
	\]
\end{theorem}

\begin{proof}
	Fix $n$. We have $\SRT^n_{+} \ured \SRT^n_{\mathbb{N}} \ured \RT^n_{\mathbb{N}}$ and $\SRT^n_{+} \ured \RT^n_{+} \ured \RT^n_{\mathbb{N}}$. Ergo, since $(\mathsf{C}_2^{(n)})^* \in \FPROB$, it suffices to show that $(\mathsf{C}_2^{(n)})^* \ured \SRT^n_{+}$, and that $\fopart \RT^n_{\mathbb{N}} \ured (\mathsf{C}^{*}_2)^{(n)}$.
	
	For the first reudction, we first note that by \Cref{fopart_rt1} and the monotonicity of the jump operator on $\sured$ we have $\mathsf{C}^{(n)}_2 \sured \mathsf{BWT}_2^{(n-1)}$. It is easy to see that $\mathsf{BWT}_2^{(n-1)} \ured \SRT^n_2$. Indeed, consider any instance of $\mathsf{BWT}_2^{(n-1)}$. We regard this as a sequence $\seq{c_{\vec{s}} : \vec{s} \in \omega^{n-1}}$ of colorings $c_{\vec{s}} : \omega \to 2$ such that $\lim_{\vec{s}} c_{\vec{s}}(x)$ exists for every $x$. Denote this limit by $c(x)$, thereby defining a coloring $c : \omega \to 2$. We define a coloring $d : [\omega]^n \to 2$ by $d(x,\vec{s}) = c_{\vec{s}}(x)$.  Then $d$ is an instance of $\SRT^n_2$ uniformly computable from $\seq{c_{\vec{s}} : \vec{s} \in \omega^{n-1}}$ with the property that $\lim_{\vec{s}} d(x,\vec{s}) = c(x)$ for all $x$. Now if $H$ is any $\SRT^n_2$-solution for $d$ then $d(H) = \lim_{\vec{s}} d(x,\vec{s})$ for every $x \in H$, so $H$ is an infinite homogeneous for $c$ with color $d(H)$. It follows that $d(H)$, which is uniformly computable from $d \oplus H \Tred\seq{c_{\vec{s}} : \vec{s} \in \omega^{n-1}} \oplus H$, is a $\mathsf{BWT}_2^{(n-1)}$-solution to $\seq{c_{\vec{s}} : \vec{s} \in \omega^{n-1}}$, as wanted. We can thus also conclude that $(\mathsf{C}^{(n)}_2)^* \ured (\SRT^n_2)^*$. But for any $k \geq 1$, the $k$-fold product $\SRT^n_2 \times \cdots \times \SRT^n_2$ is Weihrauch reducible to $\SRT^n_{2^k}$. (See, e.g., Dorais et al.~\cite{DDHMS-2016}, Proposition 2.1.) Thus, $(\SRT^n_2)^* \ured \SRT^n_+$ and so also $(\mathsf{C}_2^{(n)})^* \ured \SRT^n_{+}$. This is what was to be shown.
	
	We next show that $\fopart \RT^n_{\mathbb{N}} \ured (\mathsf{C}^{*}_2)^{(n)}$. As shown by Wang \cite[Theorem 4.2]{Wang-U} and independently by Brattka and Rakotoniaina \cite[Corollary 4.15]{BR-2017}, we have $\RT^n_{\mathbb{N}} \ured \WKL^{(n)}$. Thus $\fopart \RT^n_{\mathbb{N}} \ured \fopart \WKL^{(n)}$, and now the desired conclusion follows by \Cref{fopart_wkl}.
\end{proof}

In reverse mathematics, the first-order parts of $\SRT^n_{<\infty}$ and $\RT^n_{<\infty}$ are now fully understood. Hirst \cite[Theorem 6.4]{Hirst-1987} showed that $\RT^1_{<\infty}$ is equivalent over $\RCA_0$ to $\mathsf{B}\Sigma^0_2$, and so this is its first-order part. By results of Jockusch \cite[Lemma 5.9]{Jockusch-1972} and Simpson \cite[Theorem III.7.6]{Simpson-2009}, if $n \geq 3$ then $\RT^n_{<\infty}$ is equivalent over $\RCA_0$ to $\ACA_0$, and so the first-order part of $\RT^n_{<\infty}$ is arithmetical induction. For $n = 2$, the classification is more recent. Cholak, Jockusch, and Slaman \cite[Theorem 11.4]{CJS-2001} showed that $\SRT^2_{<\infty}$ implies $\mathsf{B}\Sigma^0_3$, while Slaman and Yokoyama \cite[Theorem 2.1]{SY-2018} showed that $\RT^2_{<\infty}$ is $\Pi^1_1$-conservative over $\RCA_0 + \mathsf{B}\Sigma^0_3$. In light of the correspondence, mentioned in \Cref{sec:intro}, between induction and bounding schemes on the one hand, and jumps of choice problems on the other, these bounds comport with those given by \Cref{FOPART_RAMSEY}. However, the theorem leaves a gap, as $(\mathsf{C}^{*}_2)^{(n)} \nured (\mathsf{C}_2^{(n)})^*$ when $n \geq 1$. (We leave this as an exercise to the reader. In both problems, instances are finite sequences of approximations to instances of $\mathsf{C}_2^*$. But in the case of $(\mathsf{C}_2^{(n)})^*$ the length of each such sequences is explicitly a part of the instance, while in the case of $(\mathsf{C}^{*}_2)^{(n)}$ the length is itself approximated.) This gap raises the following question, with which we conclude.

\begin{question}
	Can the first-order parts of $\SRT^n_{+}$, $\RT^n_{+}$, $\SRT^n_{\mathbb{N}}$, and $\RT^n_{\mathbb{N}}$ be more precisely characterized?
\end{question}

\end{document}